\documentclass[11pt]{amsart}
\usepackage[matrix,arrow,tips,curve]{xy}

\usepackage{color}
\numberwithin{equation}{section}

\makeatletter
 \renewcommand\section{\@startsection {section}{1}{\z@}%
     {-4.5ex \@plus -1ex \@minus -.2ex}%
     {2.3ex \@plus.8ex}%
    {\centering\scshape}}

\def\M{\mathcal{M}}

\def\C{\mathcal{C}}
\def\E{\mathcal{E}}
\def\F{\mathcal{F}}

\def\H{\mathcal{H}}
\def\I{\mathcal{I}}
\def\J{\mathcal{J}}

\def\mcO{\mathcal{O}}

\def\V{\mathcal{V}}

\def\X{\mathcal{X}}

\def\PP{\mathbb{P}}
\def\CC{\mathbb{C}}

\def\ZZ{\mathbb{Z}}

\def\ov#1{\overline{#1}}

\def\codim{\mathrm{codim}}

\def\Pic{\operatorname{Pic}}
\def\Hom{\operatorname{Hom}}
\def\Ext{\operatorname{Ext}}

\def\tor{\operatorname{Tor}}
\def\cliff{\operatorname{Cliff}}
\def\coker{\operatorname{Coker}}
\def\ker{\operatorname{Ker}}

\def\rk{\operatorname{rk}}
\def\ext{{\E}xt}
\def\hom{{\H}om}
\def\wt{\widetilde}

\def\oc2{\mathcal{O}_{\C_2}}

\newtheorem{theorem}{Theorem}[section]
\newtheorem{prop}[theorem]{Proposition}

\newtheorem{lem}[theorem]{Lemma}
\newtheorem{defin}[theorem]{Definition}

\newtheorem{rem}[theorem]{Remark}
\newtheorem{cor}[theorem]{Corollary}
\newtheorem{example}[theorem]{Example}
\newtheorem{claim}[theorem]{Claim}

\newcommand{\be}{\begin{equation}}
\newcommand{\ee}{\end{equation}}

\begin{document}

\title[Polarized Halphen surfaces and du Val curves]{Rank two vector bundles on polarised Halphen surfaces and the  Gauss-Wahl map for du Val curves}

\author{Enrico Arbarello}
\address{Enrico Arbarello: Dipartimento di Matematica Guido Castelnuovo, Universit\`a di Roma Sapienza
\hfill
\indent Piazzale A. Moro 2, 00185 Roma, Italy} \email{{\tt ea@mat.uniroma1.it}}

\author{Andrea Bruno}
\address{Andrea Bruno: Dipartimento di Matematica e Fisica, Universit\`a Roma Tre
\hfill \newline\texttt{}  \indent Largo San Leonardo Murialdo 1-00146 Roma, Italy} \hfill \newline\texttt{}
\email{{\tt bruno@mat.uniroma3.it}}


\maketitle
\begin{abstract}

A genus-$g$ du Val curve is a degree-$3g$ plane curve having 8 points of multiplicity $g$, one point of multiplicity $g-1$, and no other singularity. We prove that the corank of the Gauss-Wahl map of a general du Val curve of odd genus ($>11$) is equal to one. This, together with the results of \cite{ABFS}, shows that the characterisation of Brill-Noether-Petri curves with non-surjective Gauss-Wahl map as hyperplane sections of K3 surfaces and limits thereof, obtained in \cite{ABS2}, is optimal.
\end{abstract}

\section{Introduction.}

Let $C$ be a genus $g$ curve.
Recall the Gauss -Wahl map 
\be\label{G-W}
\nu=\nu_C:\bigwedge ^2 H^0(C,\omega_C)\rightarrow H^0(C,\omega_C^{\otimes 3})
\ee
defined by $s\wedge t\mapsto s\cdot dt-t\cdot ds$.
 
In \cite{ABS1} the following theorem  was proved.

\begin{theorem}\label{Main1}{\rm (Arbarello, Bruno, Sernesi).}  Let $C$ be a Brill-Noether-Petri curve  of genus $g\geq 12$. Then $C$ lies on a polarized K3 surface, or on a limit thereof, if and only if the Gauss-Wahl map is not surjective.
  \end{theorem}
  
  This theorem proves a conjecture by J. Wahl, \cite{Wahl-square}. To be precise, the original version of this conjecture made no mention of limiting K3 surfaces. Thus the  question remained to decide wether the statement of the Theorem \ref{Main1} is   optimal. To give a positive answer to this question one should produce an example  of a surface $\ov S\subset\PP^g$ with canonical sections (so that $\nu_C$ is not surjective), having  the following additional properties.
\vskip 0.1 cm  
  a) $\ov S$  is singular (i.e. has an isolated elliptic singularity), and  smoothable in $\PP^g$ (to a K3 surface).
 \vskip 0.1 cm 
 b) The general hyperplane section $C$ of $\ov S$ is a Brill-Noether-Petri curve.
 \vskip 0.1 cm 
 c) $C$ is not contained in any (smooth) K3 surface.
 \vskip 0.1 cm

In the proof of Theorem \ref{Main1} a detailed
analysis of surfaces with genus-$g$,  canonical,  sections  was carried out,  under the additional  hypothesis that these  sections should be Brill-Noether-Petri curves. 
This 
led to a list of possible examples
of smoothable surfaces in $\PP^g$ having isolated elliptic singularities,
and, possibly, Brill-Noether-Petri curves as hyperplane sections.

A very notable example, in the above list,  is the following. Take nine general points $p_1,\dots, p_9$  on $\PP^2$. A genus-$g$ du Val curve $C_0$ is a degree-$3g$ plane curve having points of multiplicity $g$ in $p_1,\dots, p_8$
and a point of multiplicity $g-1$ at $p_9$. All of these curves pass through an additional point $p_{10}$. Let $S$ be the blow-up of $\PP^2$ at 
$p_1,\dots, p_{10}$ and take the proper transform $C$ of $C_0$. The linear system $|C|$ sends $S$ to a surface $\ov S\subset \PP^g$ which is indeed a surface with canonical sections which is smooth except for a unique elliptic singularity. Moreover $\ov S$ is the limit  in $\PP^g$
of smooth K3 surfaces.  In \cite{ABFS} the following theorem was proved (see Section 1 for the definition of a nine-tuple of $k$-general points).

\begin{theorem}\label{duVal-BN}{\rm (Arbarello, Bruno, Farkas, Sacc\`a). }Suppose $p_1,\dots, p_9$ are $g$-general. Consider, as above, the du Val linear system $|C|$. Then the general element of $|C|$ is a Brill-Noether-Petri curve,  i.e
$$\mu_{0,L}:H^0(C,L)\otimes H^0(C, \omega_C\otimes L^{-1})\rightarrow H^0(C,\omega_C)$$
is injective for every line bundle $L$ on $C$.
\end{theorem}

Thus the pair $(\ov S, C)$ gives an example of surface for which properties a) and  b), above, are satisfied. The aim of this paper is to prove that also property c) is satisfied by $C$. We will in fact prove a  statement which turns out to be stronger.

\begin{theorem}\label{main-theorem}The corank of the Gauss-Wahl map for a general  du Val curve
of genus $g=2s+1>11$ is equal to one.
\end{theorem}

\begin{cor} \label{main-cor}For any odd $g>11$, 
there exist Brill-Noether-Petri curves which are (smooth) hyperplane sections of a unique surface $\ov S\subset\PP^g$ 
whose singular locus consists in 
an elliptic singularity. Moreover $\ov S$ is a limit of smooth K3 surfaces. In particular the statement of Theorem \ref{Main1} 
is optimal.
\end{cor}
The way we prove Theorem \ref{main-theorem} and Corollary \ref{main-cor} is the following. Let $g=2s+1$.
By making an appropriate choice 
of the nine points $p_1,\dots, p_9$ we construct a polarized surface $(\ov S, C)$ in $\PP^g$, as above, 
which is the direct analogue of a smooth K3 surface  $( S, C)$ in $\PP^g$, for which
$\Pic(S)=A\cdot\ZZ\oplus B\cdot\ZZ$ with $A+B=C$, where $B$ is an elliptic pencil cutting out on $C$ a $g^1_{s+1}$. We call such a surface $\ov S\subset\PP^g$ a {\it polarized Halphen surface of index $s+1$}. Halphen surfaces of index $s+1$ were already introduced by Cantat and Dolgachev in \cite{Cantat-Dolgachev} (see Section \ref{basic} below). Unlike the case of $K3$ 
surfaces, all Halphen surfaces have Picard-rank equal to $10$.  Polarized Halphen surfaces of index $s+1$ are peculiar in that they possess an elliptic pencil $|B|$  cutting out on $C$ a $g^1_{s+1}$. Following the ideas in \cite{ABS1} we prove that, in the index-($s+1$) case, the surface $\ov S$ can be reconstructed from its hyperplane section $C$ as a Brill- Noether locus of rank-2 vector bundles on $C$. Namely we establish an isomorphism
\be\label{iso-basic}
\ov S\cong M_C(2, K_C, s)
\ee
where $M_C(2, K_C, s)$ stands for the moduli space of rank-2 vector bundles on $C$ having determinant equal to $K_C$  and at least $s+2$ linearly independent sections. Let $S$ be the desingularization of $\ov S$.
The above isomorphism assigns to a point $x\in \ov S$ 
the vector bundle $E_x$ obtained as the restriction to $C$ of the unique stable torsion free sheaf $\wt \E_x$ on $S$
which is an extension of the form
$$
0\longrightarrow B\longrightarrow \wt \E_x\longrightarrow \I_x(A)\longrightarrow0
$$
with the property that $h^0(\wt \E_x)=s+2$. Such a  torsion free sheaf on $S$ belongs to the moduli space 
$M_v(S)$, with $v=(2, [C], s)$. But, unlike the case of $K3$  surfaces, where this moduli space is in fact a surface isomorphic to $S=\ov S$, in the case at hand the dimension of $M_v(S)$ is equal to five. This is one of the many
instances where the analogy between the case of $K3$ surfaces and the case of Halphen surfaces requires some care. Another instance is the geometry of the moduli space $M_C(2, K_C, s)$. Here one has to establish, a priori, that $M_C(2, K_C, s)$ has only one isolated, normal singularity.
This requires a detailed analysis of the Petri homomorphism for rank-2 vector bundles on $C$. This analysis is carried out in sub-section \ref{petri-map-s}.

Once the isomorphism (\ref{iso-basic}) is established we prove that, in the index-($s+1$) case, there is no smooth $K3$ surface containing $C$. Here we proceed by contradiction 
using the main theorem of \cite {ABS1}. If such a smooth surface $X$ existed we would find, roughly speaking,  a degenerating family $\{X_t\}$ of $K3$ surfaces having as possible central fibers both $M_v(X)$ and $\ov S$.
But since there are stable bundles on $X$ with Mukai vector $v$ (i.e. the Voisin bundles),  $M_v(X)$ is a surface with at most isolated rational singularities. By Kulikov's theorem this is not possible. This is the first step in the proof that, 
 in the index-($s+1$) case, the corank of the Gauss-Wahl map is equal to one. Once this is done the assertion about the corank is true in general. Via \cite{Wahl-sections}, this shows that for a general du Val curve $C$ of odd genus, $\ov S\in \PP^g$ is the unique surface having $C$ as canonical section.

{\bf Acknowledgements.}
It is a pleasure to  thank Giulia Sacc\`a for many interesting conversations on the subject of this paper.
Her help was decisive in wrapping up the proof of the main theorem in the last section.
We also thank Marco Franciosi for very useful conversations about Clifford's index for singular curves.

\section{Halphen surfaces and du Val curves}

\subsection{Basic definitions}\label{basic}
 Let $\tau: S'\to \PP^2$ be the blow up of $\PP^2$ at 9 (possibly infinitely near) points $p_1,\dots, p_9\in \PP^2$. Assume there is a unique  cubic $J_0$ through $p_1,\dots, p_9\in \PP^2$. Let $J'$ be the proper transform of $J_0$ in $S'$.
$$
\{J'\}=|-K_{S'}|\,,\qquad J'^2=0
$$
\begin{defin}
$S'$ is said to be unnodal if it contains no (-2)-curves.
\end{defin}
\vskip 0.2 cm
\begin{defin} a) The points 
$p_1,\dots, p_9$ are said to be $k$-Halphen general if 

 $h^0(J',  \mcO_{J'}(hJ'))=0$, for $1\leq h\leq k$, i.e if 

$h^0(S',  \mcO_{S'}(hJ'))=1$, for $1\leq h\leq k$.
\vskip 0.1 cm 
b) The points $p_1,\dots, p_9$ are said to be $k$-general if they are $k$-Halphen general, and  if $S'$ is  unnodal.
\end{defin}
\vskip 0.2 cm
{\bf Remark on terminology}. In \cite{ABFS}, a set of nine points for which $S'$ is unnodal was called Cremona general and a set of $k$-Halphen general points was called 
$3k$-Halphen general. We decided to follow \cite{Cantat-Dolgachev} which preceded \cite{ABFS}.

\vskip 0.2 cm
\begin{example}{\rm( see \cite{ABFS})}
The points $$\aligned &(-2,3),(-1,-4),(2,5),(4,9),(52,375), (5234, 37866),(8, -23), \\&(43, 282), \Bigl(\frac{1}{4}, -\frac{33}{8} \Bigr)
\endaligned$$
are $k$- general for every $k$.
\end{example}

Let $\ell\subset S'$ be the proper transform of a line in $\PP^2$, and $E_1,\dots, E_9$ the exceptional divisors of $\tau$ so that
$$
J'=3\ell-E_1-\cdots-E_9
$$
 
\begin{defin}\label{dv} A du Val curve of genus $g$ is a degree-$3g$ plane curve $C_0$ having a point of multiplicity $g$ in   $p_1,\dots, p_8$, a point of multiplicity $g-1$ in   $p_9$, and no other singularities.
On $S'$ we have:
$$
C'=3g\ell-gE_1-\cdots-gE_8-(g-1)E_9\,,\qquad C\cdot J'=1\,,\quad \dim|C'|=g\,.
$$
\end{defin}

{\bf Remarks. }We work on $S'$.

{\it 
{\rm a)} As soon as  $p_1,\dots, p_9$ are $1$-Halphen general, a du Val curve exists.

{\rm b)} $C'\cap J'=\{p_{10}\}$, so that $p_{10}$  is a fixed point for $|C|$. The base point 
$$
p_{10}=p_{10}(g)
$$
plays an important role.

{\rm c)}
$C'=gJ'+E_9$, so that $|C'-J'|$ is a du Val linear system of genus $g-1$.
Thus the linear system $|C'|$ contains a reducible element formed by a du Val curve of genus $g-1$ and the elliptic curve $J'$ meeting at $p_{10}$.
}

\vskip 0.2 cm
Now blow up $S'$ at $p_{10}$
and use the following notation:
\be\label{notat1}
\aligned
\sigma: &S\longrightarrow S'\,,\quad\text{is the blow up}\\
E_{10}&=\sigma^{-1}(p_{10})\,,\\
J'&=\sigma^{-1}(J')\,,\qquad\text{(with a slight abuse of notation)}\\
J&=\text{proper transform of} \,\quad J'\\
C&=\text{proper transform of}\quad C'\,,\qquad\\
\endaligned
\ee
Then
\be\label{notat2}
\aligned
J'&=J+E_{10}\\
C&=3g\ell-gE_1-\cdots-gE_8-(g-1)E_9-E_{10}\,,\\
 J^2&=-1\,,\qquad C\cdot J=0\,,\qquad\dim |C|=g
\endaligned
\ee

$$
\aligned
\phi&=\phi_{|C|}: S\longrightarrow \ov S\subset \PP^g\,,\\
& S\setminus J\quad\overset\phi\cong\quad \ov S\setminus\{pt\}\\
&\phi(J)=\{pt\}=\{\text{an elliptic singularity of}\,\,\,\ov S\}
\endaligned
$$

\begin{defin} The pair $(\ov S, C)$ is a polarized Halphen surface (of 
genus $g$).
\end{defin}

\begin{prop} {\rm (Arbarello, Bruno, Sernesi \cite{ABS2}).}
$\ov S$ is a limit of smooth K3 surfaces in $\PP^g$.
\end{prop}

\begin{theorem} {\rm (Nagata).}\label{nagata}  Suppose  $p_1,\dots, p_9$ are  $k$-general. Let $D=d\ell-\sum\nu_iE_i$ be an effective divisor
with $d\leq 3k$ and 
such that $D\cdot J'=0$. Then $D=mJ'$ for some $m$.
\end{theorem}

\begin{defin} {\rm(Cantat-Dolgachev \cite{Cantat-Dolgachev}).}  Let $m$ be a positive integer.  Then $S'$ is a Halphen surface of index $m$ if $p_1,\dots, p_9$
are $k$-Halphen general for $k\leq m-1$ but are not $m$-Halphen general.

Equivalently, if:
$$
\aligned
i)&\qquad \mcO_{J'}(mJ')=\mcO_{J'}\,,\\
ii)&\qquad \mcO_{J'}(hJ')\neq\mcO_{J'}\,,
\quad 1\leq h\leq m-1
\endaligned
$$
If $S'$ is a Halphen surface of index $m$, 
we will also say that the blow up $S$ of $S'$ at $p_{10}(g)$
is a Halphen surface surface of index $m$
\end{defin}

\section{Genus-$(2s+1)$ polarised Halphen surfaces of index $(s+1)$}

From now on  $S'$ is a Halphen surface of index $s+1$, with $s \geq 6$, and $C$ is a du Val curve 
of genus $g=2s+1$ on $S'$. We  refer to  notation (\ref{notat1}) and  (\ref{notat2}), and we denote by $B$ the pencil $B=(s+1)J'$ on $S$. On $S$ we have:

$$
\aligned
J'&=J+E_{10}\,,\qquad J'^2=0 \,,\qquad J'\cdot J=0 \,,\qquad K_S=-J=-J'+E_{10}\\
F&:=E_9-E_{10}\,,\qquad F^2=-2\,,\qquad J'\cdot F=1\\
C&=gJ'+F \,,\qquad g=2s+1  \,,\qquad C_{|J}=\mcO_J \\
A&=sJ'+F=(g-s-1)J'+F\\
B&=(s+1)J'\,,\qquad C=A+B\\
\endaligned
$$
We set 
$$
B_{|C}=\xi\,,\qquad A_{|C}=\eta\
$$
We also assume that $J$ is smooth.

\subsection{Preliminary computations.}

\begin{prop}\label{calcoli} Suppose $S$ is a Halphen surface of index $(s+1)$. Then we have:

{\rm i)} $h^0(S,\mcO_S( B))=2\,,\qquad h^1(S,\mcO_S( B))=1\,,\qquad h^2(S,\mcO_S( B))=0$,

{\rm ii)} $h^0(S,\mcO_S( 2B))=3\,,\qquad h^1(S,\mcO_S(2 B))=2\,,\qquad h^2(S,\mcO_S( 2B))=0$,

{\rm iii)} $h^0(S,\mcO_S( 2B-J))=2\,,\qquad h^1(S,\mcO_S( 2B-J))=1\,,\qquad h^2(S,\mcO_S(2B -J))=0$,

{\rm iv)} $h^0(S,\mcO_S( A-B))=0\,,\qquad h^1(S,\mcO_S( A-B))=1\,,\qquad h^2(S,\mcO_S(A- B))=0$,

{\rm v)}  $h^0(S,\mcO_S( B-A))=0\,,\qquad h^1(S,\mcO_S( B-A))=1\,,\qquad h^2(S,\mcO_S( B-A))=0$,

\end{prop}
\begin{proof}
By hypothesis, the pencil $(s+1)J'$ on $S'$ is base-point free, and since $S$ is Halphen of index $(s+1)$ we have
$$
h^0(S, B)=h^0(S', \mcO_{S'}((s+1)J'))=2\,
$$
On the other hand $h^2(S, B)=0$, and from the Riemann-Roch theorem on $S$  we get i).

We also have $h^2(S, \mcO_{S}(2B-J))=h^2(S, \mcO_{S}(2B))=0$. 
Since $ \mcO_{J}(2B)=\mcO_J$, and $|B|$ is base-point free,
from the exact sequence
$$
0\to \mcO_{S}(2B-J)\longrightarrow \mcO_{S}(2B)
\longrightarrow \mcO_{J}(2B) \to 0
$$
we get that 
$h^0(S, \mcO_{S}(2B-J))=h^0(S, \mcO_{S}(B))-1$
and $h^1(S, \mcO_{S}(2B-J))=h^1(S, \mcO_{S}(B))-1$.

From
$$
0\longrightarrow \mcO_{S}(B) \longrightarrow \mcO_{S}(2B)
\longrightarrow\mcO_{B}(2B)\longrightarrow 0
$$
and again base-point freeness of $|B|$, we deduce
$$
h^0(S, \mcO_{S}(2B))=3\,,\quad h^1(S, \mcO_{S}(2B))=2\,,\quad h^2(S, \mcO_{S}(2B))=0\,,\quad
$$
and
$$
h^0(S, \mcO_{S}(2B-J))=2\,,\quad h^1(S, \mcO_{S}(2B-J))=1\,,
$$
yielding  ii) and iii).
We finally prove iv) and v).
We have
$$
B-A=J'-F=J+2E_{10}-E_9\,,\qquad B-A-J=2E_{10}-E_9
$$
so that 
$$
h^0(S, \mcO_{S}(A-B))=h^2(S, \mcO_{S}(B-A-J))=0 \qquad h^2(S, \mcO_{S}(B-A))=h^0(S, \mcO_{S}(A-B-J))=0
$$
From the Riemann-Roch theorem
$$
\chi(S, \mcO_{S}(B-A))=h^0(S, \mcO_{S}(B-A))-h^1(S, \mcO_{S}(B-A))=1+\frac{(B-A)^2}{2}=1-2=-1
$$
Since  $J$ is irreducible, from
$$
0\longrightarrow\mcO_S(J-E_9)\longrightarrow\mcO_S(J-E_9+2E_{10})\longrightarrow \mcO_{2E_{10}}(J-E_9+2E_{10})\longrightarrow0
$$
we get
$$
h^0(\mcO_S(B-A))=h^0(\mcO_S(J-E_9))=0
$$
so that
$$
h^1(\mcO_S(B-A))=1
$$
Finally $h^0(S, \mcO_S(A-B))=h^2(S, \mcO_S(A-B))=0$, and, again from the Riemann-Roch theorem we get $h^1(\mcO_S(A-B))=1$.
\end{proof}

\begin{prop}\label{A} 
$|A|$ is a base point free linear system of  du Val curves of genus $s$, whose general element is Brill-Noether general.
The pair $(\ov S, A)$ is a polarized Halphen surface of genus $s$ which is cut out by quadrics. In particular:

{\rm i)} $h^0(S, \mcO_S(A))=s+1\,,\qquad h^1(S, \mcO_S(A))=1\,,\qquad h^2(S, \mcO_S(A))=0$,

{\rm ii)}  $h^0(S, \mcO_S(A-J))=s\,,\qquad h^1(S, \mcO_S(A-J))=0\,,\qquad h^2(S, \mcO_S(A-J))=0$,

{\rm iii)} $h^0(S, \mcO_S(2A))=4s-2\,,\qquad h^1(S, \mcO_S(2A))=1\,,\qquad h^2(S, \mcO_S(2A))=0$.

Finally, if $A \in |A|$ is a reducible member, then $A=kJ'+(k-1)E_{10}+A_{s-k}$ where $k \geq 1$
and $A_{s-k}=(s-k)J'+E_9$ is a du Val integral curve of genus $s-k$. (The linear system $|A_{s-k}|$ will have a base point $p_{10}(s-k)$.)
\end{prop}
\begin{proof}
Since $B_{|J}=C_{|J}=\mcO_J$
we also have
\be\label{a-rstr-j}
A_{|J}=\mcO_J
\ee
Hence $|A|$ is a base-point-free linear system of du Val curves of genus $s$. 
A consequence of (\ref{a-rstr-j}) is that
$$
p_{10}(g)=p_{10}(s)
$$
A general member of $|A|$  is Brill-Noether general by Theorem \ref{duVal-BN}, since $S$ is unnodal and of index $s+1$.
As for all du Val curves, we get i) and ii) from the exact sequence:
$$
0\longrightarrow \mcO_S(A-J) \longrightarrow \mcO_S(A)
\longrightarrow \mcO_J(A) \longrightarrow 0
$$
 
As far as  iii) is concerned,  consider the sequence
$$
0\longrightarrow \mcO_S(2A-J)
\longrightarrow \mcO_S(2A)
\longrightarrow \mcO_J(2A)
\longrightarrow0
$$
Since $\mcO_J(2A)=\mcO_J$, and $h^2(S, \mcO_S(2A-J))=h^2(S, \mcO_S(2A))=0$, we get 
 $h^1(S, \mcO_S(2A))\neq0$. On the other hand, if we restrict $2A$ to $A$ 
and if we notice that $2A_{|A}$ is not special, we get a surjection
$H^1(S, \mcO_S(A)) \longrightarrow H^1(S, \mcO_S(2A))$. Recall from i) that $h^1(S, \mcO_S(A))=1$, so that  $h^1(S, \mcO_S(2A))=1$. From
the Riemann-Roch theorem we then obtain
$h^0(S, \mcO_S(2A))=2+2A^2=4s-2$.
Since $s \geq 5$, and since the curve  $A\subset \PP^{s-1}$ is Brill-Noether general, it must be cut out by quadrics. On the other hand, $A$ is the hyperplane section of the surface 
  $\ov S \subset \PP (H^0(S, \mcO_S(A)))=\PP^s$, so that $\ov S$ must be cut out by quadrics as well. Let us now come to the last point of the Proposition.
Let $A \in |A|$ a reducible element and consider the blow-up 
$\sigma: S \longrightarrow S'$
of $S'$ at $p_{10}$.
Let us write $A=\sigma^*(A') -E_{10}$, and let $Y' \subset S'$ be the irreducible component of $A'$ intersecting $J'$, i.e. $J'\cap Y'=p_{10}$.
We then  have  $(A'-Y') \cdot J'=0$, and since the  points $p_1, \ldots, p_9$ are $s$-general, from
Theorem \ref{nagata} we get that $(A'-Y')=kJ'$ for some $k \geq 1$. On $S$ we have:
$$
sJ'+F= A=kJ'+\sigma^*(Y') -E_{10}.
$$
where $\sigma^*(Y') -E_{10}$ is effective, so that $s>k$.
Thus $A_{s-k}:=\sigma^*(Y') -E_{10}$ is a du Val curve of genus $s-k$. 

\end{proof}
\begin{prop}\label{qhull} We have: 
$B_{|C}=\xi=g^1_{s+1}$ and  $A|_C= \eta=g^s_{3s-1}$. Both $\xi$ and $\eta$ are 
 base point free and 
the quadratic hull of $\phi_{\eta}(C) \subset \PP(H^0(C, \mcO_C(\eta)))=\PP(H^0(S, \mcO_S(A)))=\PP^s$ is $\ov S$.
\end{prop}
\begin{proof}
Since $B-C=-A$, 
from the exact sequence:
$$
0\longrightarrow \mcO_S(B-C)
\longrightarrow \mcO_S(B)
\longrightarrow \mcO_C(\xi)
\longrightarrow0
$$
using Proposition \ref{A} ii), and the Riemann-Roch, and Serre's duality theorems
we obtain  $h^0(C, \mcO_C(\xi))=2$.
From this we also deduce that $h^0(S,\mcO_S(A))=h^0(C, \mcO_C(\eta))=s+1$.
Since $|A|$ is a du Val system of curves of genus $s$ whose general member is Brill-Noether general, 
$|A|$ is also quadratically normal on $S$ and the image $\ov S \subset \PP (H^0(S,\mcO_S( A)))$
is generated by quadrics. Consider the exact sequence
$$
0\longrightarrow \mcO_S(2A-C)
\longrightarrow \mcO_S(2A)
\longrightarrow \mcO_C(2\eta)
\longrightarrow0
$$
From Proposition \ref{calcoli} iv),  Proposition \ref{A} iii) and the Riemann-Roch theorem on $C$, we get that the restriction map induces an isomorphism $H^0(S, \mcO_S(2A))\cong H^0(C, \mcO_C(2\eta))$ and that
$h^0(S, \mcO_S(2A))=h^0(C, \mcO_C(2\eta))=4s-2$.
From the surjective homomorphism
$$
S^2(H^0(C, \mcO_C(\eta))) =S^2(H^0(S, \mcO_S(A)))\longrightarrow H^0(S, \mcO_S(2A))=H^0(C, \mcO_C(2\eta))
$$
we see
that the intersection of all the quadrics containing $\phi_{\eta}(C)$, i.e. the quadratic hull of $(C, \eta)$, is $\ov S$.
\end{proof}

\subsection{On some extensions of torsion free sheaves on $S$.}

For $x \in S$ we want to study coherent sheaves on $S$ which are extensions:

\be\label{ext-abx}
0\longrightarrow\mcO_S(B) \longrightarrow\E_x\longrightarrow \I_x(A)\longrightarrow0
\ee

Such extensions are, a priori, only torsion free sheaves on $S$, and are classified by $\Ext^1(\I_x(A), \mcO_S(B))$.
From the local to global spectral sequence for the $\Ext$-functors we get an exact sequence:

$$
0\to H^1(\hom(\I_x(A),B)) \to \Ext^1(\I_x(A), \mcO_S(B))  \to H^0(\ext^1(\I_x(A), B)) \to H^2(\hom(\I_x(A), B))\to 0
$$

\begin{lem}\label{extI_x}
We have $\hom(\I_x(A), \mcO_S(B))=\hom(\mcO_S(A),\mcO_S( B))=\mcO_S(B-A)$ and $\ext^1(\I_x(A), \mcO_S(B))=\CC_x$, so that
$H^2(\hom(\I_x(A), \mcO_S(B)))=0$.
 Moreover, $\dim \Ext^1 (\I_x(A), \mcO_S(B))=2$
\end{lem}
\begin{proof}

Since $S$ is regular, for $x \in S$, a locally free resolution of $\mcO_x$ is given by the Koszul complex
$$ 0 \to  \mcO_S \longrightarrow \mcO_S^2 \longrightarrow\mcO_S\longrightarrow\mcO_x \to 0$$
Applying $\hom(-, \mcO_S(B))$ and taking cohomology we get
$$
\ext^1(\mcO_x(A), \mcO_S(B))=0
$$
From $$ 0 \longrightarrow \I_x(A) \longrightarrow \mcO_S(A) \longrightarrow \mcO_{x}(A) \to 0
$$
we get
$$
\hom(\I_x(A), \mcO_S(B))=\hom(\mcO_S(A), \mcO_S(B))=\mcO_S(B-A) 
$$
It follows that $H^2(\hom(\I_x(A), \mcO_S(B)))=H^2(\hom(\mcO_S(A),\mcO_S(B)))=H^2(\mcO_S(B- A))=0$ from Proposition \ref{calcoli} v). 
Always from Proposition \ref{calcoli} v) we get 
$$
0\longrightarrow \CC =H^1(\hom(\I_x(A), \mcO_S(B))) \longrightarrow \Ext^1(\I_x(A), \mcO_S(B))  \longrightarrow H^0(\ext^1(\I_x(A), \mcO_S(B))) =\CC\longrightarrow 0
$$
\end{proof}
Notice that $H^1(\hom(\I_x(A), \mcO_S(B)))$ is naturally identified with 
$$
H^1(\mcO_S(B-A))=\Ext^1(\mcO_S(A),\mcO_S(B))
$$
We will show next that for every $x \in S$ the space of  isomorphism classes of non-split extensions (\ref{ext-abx}), which
can be identified with $\PP (\Ext^1(\I_x(A), \mcO_S(B)))$, contains exactly one extension which is not locally free and exactly
one extension with $s+2$ sections.

The next result is  a direct consequence of Theorem 5.1.1 in \cite{Huybrechts-Lehn}:

\begin{lem}\label{nlf}Extensions of the form (\ref{ext-abx}) which are not locally free
are parametrised by $\Ext^1(\mcO_S(A),\mcO_S(B))$. In particular, for every $x \in S$ there is, up to scalar, a unique non-split extension which is not locally free.
\end{lem}

\begin{proof}
Following Theorem 5.1.1 in \cite{Huybrechts-Lehn}, and using Proposition \ref{calcoli} v),  we see  that the cohomology group $H^0(S, \mcO_S(A-B-J))$ vanishes,  so that the Cayley-Bacharach property holds.
From the proof of the  Theorem 5.1.1 in \cite{Huybrechts-Lehn}, it then follows  that non-split extensions which are not locally free are all obtained from the
unique non-split extension 
\be\label{ext-ab}
0 \to \mcO_S(B) \to \E \to \mcO_S(A) \to 0
\ee
and sit in the diagram:
\be\label{non-loc-free}
\xymatrix
{
&&0\ar[d]&0\ar[d]\\
0\ar[r]&\mcO_S(B)\ar[r]\ar@{=}[d]&\E_x\ar[r]\ar[d]&\I_x(A)\ar[d]\ar[r]&0\\
0\ar[r]&\mcO_S(B)\ar[r]&\E\ar[r]\ar[d]&\mcO_S(A)\ar[r]\ar[d]&0\\
&&\CC_x\ar@{=}[r]\ar[d]&\CC_x\ar[d]\\
&&0&0\\
}
\ee
\end{proof}

We will denote by $y=p_{10}(s-1) \in J$  the unique base point of the  du Val linear system $|A-J|$.
The point $y$ will be relevant also in the proof of Lemma \ref{s2eb} below, and we will find it convenient to call it $p_{11}$.

\begin{lem} \label{h12} Consider an extension of type (\ref{ext-abx}).

a)  For every $x\in S$, we have $h^2(S, \E_x(-J))=0$.

b) If $x\neq y$, then $h^1(S, \E_x(-J))=0$ and $h^0(S, \E_x(-J))=s$.

c) If $x\neq y$, the restriction map $H^1(S, \E_x) \longrightarrow H^1(J,  {\E_x}{|_J})$
is an isomorphism,  and we have an exact sequence
$$
0\longrightarrow \CC^s\longrightarrow H^0(S, \E_x)\longrightarrow H^0(S, {\E_x}_{|J})\longrightarrow0
$$
d) If $x=y\in J$ then $h^1(S, \E_x(-J))=1$ and $h^0(S, \E_x(-J))=s+1$

\end{lem}

\begin{proof}
 a) For any $x \in S$, consider the exact sequence
\be\label{E-J}
0 \longrightarrow
\mcO_S(B-J) \longrightarrow
\E_x(-J) \longrightarrow
\I_x(A-J) \longrightarrow 0
\ee
Since $B-J$ and $A-J$ are effective, for any $x \in S$, we have $h^2(S, \E_x(-J))=0$.

b) From the exact sequence
$$
 0 \to I_x(A-J) \to \mcO_S(A-J) \to (A-J)_{|x} \to 0
 $$
and from Proposition \ref{A} ii), we get that $h^1(S, I_x(A-J))=0$, if and only if $x \neq y$.
Since $S$ is of index $(s+1)$, we have that $h^1(S, \mcO(B-J))=0$,  and hence $h^1(S, \E_x(-J))=0$, if and only if $x \neq y$.

c) Follows at once from a) and b).

d) If $x=y$, we have $h^0(I_y(A-J))=s$  and $h^1(I_y(A-J))=1$,
so that $h^1(S, \E_y(-J))=1$.  Since $\chi(S, \E_y(-J))=s$, we must  have $h^0(S, \E_y(-J))=s+1$,
\end{proof}

\begin{prop}\label{wt}

 For every $x \in S$ there exists a unique, up to a scalar, non-split extension (\ref{ext-abx})  such that
$h^0(S, \E_x)=s+2$. If $ x \notin J$ such an extension is a locally free sheaf and  $ {\E_x}_{|J}=\mcO_J^2$.
If $x \in J$ such an extension is not locally free.
\end{prop}

\begin{proof}
We first observe that $\chi (S, \E_x)=\chi(S, \E_x(-J))=s$
for every $x \in S$.
Since from Propostions \ref{calcoli} i) and \ref{A} i) we have $h^2(S, \E_x)=0$, it follows that
$h^0(S, \E_x)=s+2$ if and only if $h^1(S, \E_x)=2$.

 {\bf Case $x\notin J$}.
 
In this case, restricting an extension of the form  (\ref{ext-abx})
to $J$ we get an extension 
\be\label{ext-j}
0 \to \mcO_J \to {\E_x}{|_J} \to \mcO_J \to 0
\ee
yielding  a homomorphism
$$
\rho: \Ext^1_S(\I_x(A), \mcO_S(B))\longrightarrow\Ext^1_J(\mcO_J, \mcO_J)=H^1(J, \mcO_J)
$$
This homomorphism is  surjective. Indeed, look at the subspace
$$
 H^1(S, \mcO_S(B-A)) \cong  H^1(\hom(\I_x(A), \mcO_S(B))) \subset Ext^1_S(\I_x(A), \mcO_S(B))
 $$ 
on this subspace $\rho$ induces the restriction map 
$$
 H^1(S, \mcO_S(B-A)) \longrightarrow H^1(J,   \mcO_J(B-A))=H^1(J,   \mcO_J)
 $$
induced by $0 \to \mcO_S( B-A-J )\to\mcO_S( B-A) \to \mcO_J(B-A) \to 0$, which  is an isomorphism.
From (\ref{ext-j}), we see that 
$h^1(S, {\E_x}{|_J}) =2,\,\text{or}\, 1$, depending on whether  the extension class of ${\E_x}{|_J}$ is zero or non-zero. It follows 
that, up to a non-zero scalar,  there exists a unique extension $\E_x$ whose class is in the kernel of $\rho$. For such an extension we have
 $h^0(S, \E_x) = s+2$.
By Lemma (\ref{nlf}), the extension $\E_x$ is locally free because it does not come from a non zero element
of $H^1(S, \mcO_S(B-A))$. We interrupt the proof of  Proposition \ref{wt} to prove the following lemma.
\begin{lem}\label{e-prime}
 Consider the unique non-split extension (\ref{ext-ab}).
Then $h^0(S, \E)=s+2$.
\end{lem}
\begin{proof}Indeed, consider a diagram (\ref{non-loc-free}) where $x\notin J$,  and where (\ref{ext-j}) is non-split. We get  $h^1(S, {\E_x}{|_J}) =1$, so that, by Lemma \ref{h12}, $h^1(S, {\E_x}) =1$,
and, as a consequence, $h^0(S, {\E_x}) =s+1$.  From  diagram (\ref{non-loc-free})
it follows that $h^0(S, {\E}) \leq s+2$. Now look at the same diagram in the case in which
(\ref{ext-j}) is split. Then $h^0(S, \E_x) = s+2$, so that $h^0(S, {\E}) \geq s+2$.
\end{proof}

Let us resume the proof of Proposition \ref{wt}.

{\bf Case $x\in J$}.

In this case, from a local computation, we get that $\tor^1(\mcO_x, \mcO_J)=\CC_x$,
and $\tor^1(\I_x, \mcO_J)=0$. This means that
$$0 \to \mcO_J \to {\E_x}_{|_J} \to  {I_x}{|_J} \to 0$$ is exact and that there is a 
an exact sequence 
$$ 0 \to \tor^1(\mcO_x, \mcO_J) \cong \CC_x \to {I_x}{|_J} \to \mcO_J(-x) \to 0$$
Suppose first that $\E_x$ is locally free. Then ${\E_x}{|_J}$ is locally free as well, and by composition
we get a surjection of locally free sheaves 
$$
{\E_x}{|_J} \to \mcO_J(-x) \to 0
$$
Hence we have an extension 
$$
0 \to \mcO_J(x) \to  {\E_x}{|_J} \to \mcO_J(-x) \to 0
$$ 
which splits
since $h^1(J, \mcO_J(2x))$ vanishes. In particular $h^1(J, {\E_x}{|_J}) <2$, for all $x \in J$, whenever $\E_x$ is locally free.
Let us then consider an extension  of the form  (\ref{ext-abx}) which is not locally free. By diagram (\ref{non-loc-free}), the restriction ${\E_x}{|_J}$
is not torsion free, and since $\mcO_J$ is torsion free, the torsion subsheaf of ${\E_x}{|_J}$ is 
contained and hence isomorphic to $\CC_x$, the torsion subsheaf
of $ {I_x}{|_J}$.
Let $E'$ be the torsion free quotient of ${\E_x}{|_J}$. Then we have an exact sequence
$$
 0 \to \CC_x \to {\E_x}{|_J} \to E' \to 0
 $$ 
 and an extension $0 \to \mcO_J \to E' \to \mcO_J(-x)\to 0$
which is necessarily split because $h^1(J, \mcO_J(x))=0$. Then  $h^1(J, {\E_x}{|_J})=
h^1(S, {E'})=2$. From Lemma \ref{h12} a), we get $h^1(S, {\E_x})\geq2$, so that
$h^0(S, {\E_x})\geq s+2$. But $\E_x$ is a subsheaf of $\E$, thus, by Lemma \ref{e-prime},
we conclude that $h^0(J, {\E_x})=s+2$.

\end{proof}

\begin{defin}\label{wtdef}
For every $x \in S$, we will denote by
\be\label{etilde}
e_x:\qquad 0\longrightarrow\mcO_S(B)\longrightarrow\wt{\E}{}_x\longrightarrow \I_x(A)\longrightarrow 0
\ee
the unique non-split extension with $h^0(S, \wt\E_x)=s+2$,  given by Proposition (\ref{wt}).
\end{defin}\label{ate}
We now relativize this picture. We let $T$ be a copy of $S$. Let $p$ and $q$ be the projections of $S\times T$ to $S$ and $T$ respectively. Let $\Delta\subset S\times T$ be the diagonal. It is straightforward to see that $\ext^1_{S\times T}(\I_\Delta(p^*A), p^*(B))$
is a rank 2  locally free sheaf on $T$ whose fiber over $x\in T$ is $\Ext^1(\I_x(A), B)$. 
We denote by ${\bf P}$ the associated $\PP^1$-bundle. The association $x\mapsto [e_x]$ defines a section  $e: T\to {\bf P}$. 
Let $\phi$ and $\psi$ be the  projections from $S\times{\bf P}$ to ${S\times T}$ and ${\bf P}$, respectively. From Corollary 4.4 in \cite{Lange} we get
a universal extension over $S\times {\bf P}$
$$
0\longrightarrow\phi^*(p^*\mcO_S(B))\otimes\psi^*(\mcO_{\bf P}(1))\longrightarrow \E_{\bf P}\longrightarrow \phi^*(\I_\Delta(p^*\mcO_S(A)) \longrightarrow0
$$
If we identify $T$ with its image in ${\bf P}$ via the section $e$, we get an extension over $S\times T$
\be\label{ue}
0\longrightarrow p^*\mcO_S(B)\longrightarrow \wt\E_{T}\longrightarrow (\I_\Delta(\mcO_S(A))) \longrightarrow0
\ee
whose fiber over $x\in T$ is $e_x$ ( as in Definition \ref{wtdef}).

\vskip 0.5 cm 
\subsection{Stable vector bundles on $C$ with $s+2$ linearly independent  sections.}
In this section, for every $x \in S$, we consider, the restriction to $C$ of the sheaf sheaf $\wt{\E}{}_x$ defined in (\ref{wtdef}).
Let 
$$
E_x:={\wt{\E}_x}|_{C}
$$
We observed that if $x \notin J$, the sheaf $E_x$ is a locally free sheaf.  We will show that $h^0(C,E_x)=s+2$ and that $E_x$ is stable.

\begin{prop}\label{sections}
For all $x \in S$, $h^0(C, E_x)=s+2$. 
If $x,y \in J$ we have $E_x =E_y$.
\end{prop}
\begin{proof}
We first consider the case $x \notin J$. Since, in this case, $\wt \E_x$ is a locally free
sheaf of rank $2$ with determinant $C$, from Serre's duality we have :

$$
\wt\E_x^\vee=\wt\E_x(-C) \,,\qquad h^i(S, \wt\E_x^\vee)=h^i(S, \wt\E_x(-C))= h^{2-i}(S, \wt\E_x(-J))
$$

From Lemma (\ref{h12}) we have $h^i(S, \wt \E_x(-J)=0$, for $i \geq 1$, and then $h^i(S, {\wt\E}_x(-C))=0$, for $i \leq 1$.
We conclude the case $x \notin J$  by looking at the exact sequence
$$
0\longrightarrow \wt \E_x(-C) \longrightarrow \wt\E_x \longrightarrow E_x \longrightarrow 0
$$

Consider now the case $x \in J$. In this case $\wt\E_x$ is not locally free, and sits in a diagram
(see the proof of Lemma (\ref{nlf})):
\be\label{tors-free}
\xymatrix
{
&&0\ar[d]&0\ar[d]\\
0\ar[r]&\mcO_S(B)\ar[r]\ar@{=}[d]&\wt\E_x\ar[r]\ar[d]&\I_x(A)\ar[d]\ar[r]&0\\
0\ar[r]&\mcO_S(B)\ar[r]&\E\ar[r]\ar[d]&\mcO_S(A)\ar[r]\ar[d]&0\\
&&\CC_x\ar@{=}[r]\ar[d]&\CC_x\ar[d]\\
&&0&0\\
}
\ee
where $\E$ is the unique non-split extension of $\mcO_S(A)$ by $\mcO_S(B)$. 
In particular, since $J \cap C= \emptyset$  for all $x \in J$ we have that $E_x=\E|_C$.
From Remark \ref{e-prime}, we know that $h^0(S,\E)=s+2$.
Since $\E$ is a locally free
sheaf of rank $2$ with determinant $C$, we have from Serre's duality:

\be\label{vari-iso-e}
{\E}^\vee=\E(-C) \qquad h^i(S, {\E}^\vee)=h^i(S, \E(-C))= h^{2-i}(S, \E(-J))
\ee

We then need to prove that 
\be\label{he1}
h^i(S, \E(-C))=0\,,\quad  i \leq 1
\ee
i.e.  we need to show that 
\be\label{he2}
h^i(S, \E(-J))=0\,,\quad i \geq 1. 
\ee

This is done via a
 a computation which is similar but easier than the one in Lemma (\ref{h12}). We leave this to the reader.
\end{proof}

In order to prove stability of the locally free sheaf $E_x$ for all $x \in S$ we first prove an analogue  of Lemma 4.3 of \cite{ABS1}.

\begin{lem}\label{stab}Let $D\subset C\subset S$ be a finite closed subscheme of length $d\ge 1$.
Assume that
\begin{equation}\label{E:ineq1}
h^0(S, \mathcal{I}_D(A))\geq \mathrm{max} \left\{3,s- \frac{d-1}{2} \right\}
\end{equation}
 Then $d=1$.
\end{lem}

\begin{proof}
 We   view $S$ as embedded in $\PP^s=\PP H^0(S,\mathcal{O}(A))$. Consider  a   hyperplane $H$ passing through $D$, i.e defining a non-zero element of $H^0(S, \mathcal{I}_D(A))$. 
We set $A=H\cap S$. If $A$ is integral we proceed exactly as in Lemma 4.3 of \cite{ABS1} and we obtain that $\cliff (A) \leq 1$. Since $A$ is a du Val linear system, $A|_A$ is very ample, and from
Theorem A in \cite{Franciosi-Tenni} it follows that $\cliff (A) = 1$. Let $D^*$ be the adjoint divisor to $D$, 
with respect to a general section $s\in H^0(A, K_A)$, in the sense of Definition 2.8 in \cite{Franciosi-Tenni}. If we set $L=K_A$, $M_1=D$, $M_2=D^*$ in the proof of 
the nonvanishing theorem of Green-Lazarsfeld (in the appendix of   \cite{Green-Koszul}) we get that
the Koszul cohomology group $K_{s-3,1}(A) \neq 0$. From duality we obtain$K_{1,2}(A) \neq 0$, i.e. the canonical model of $A$ is not cut out by quadrics. But, by Proposition \ref{A},  the surface  $\phi_A(S)$ is cut out
by quadrics  and $A$ is a hyperplane section of $\phi_A(S)$. This is impossible. This shows  that $d=1$.  

Suppose that $A$ is not integral. Then, from Lemma (\ref{A}), $A=kJ'+(k-1)E_{10}+A_{s-k}$ where $k \geq 1$
and $A_{s-k}=(s-k)J'+E_9$ is a du Val integral curve of genus $s-k$. Such linear system has a base point at $r=p_{10}(s-k) \in J$.
Notice that $r$ does not lie on $D$ as $J$ and $C$ are disjoint and $D \subset C$.

Let $q=q_{10}=E_{10}\cap C$. and write 
$$
D=D'+lq\,,\qquad \deg D=d\,,\quad k\geq 1\,,\quad l \leq k-1
$$
We have
\be\label{E:ineq1-2}
h^0(S, \I_D(A))=
h^0( S, \I_{D'}(A_{s-k}))
\geq \max\left(3, s-\frac{d-1}{2}\right)
\ee

In particular we observe that  $s-k \geq 3$.
 We may  view $D'$ as a subscheme of the integral curve $A_{s-k}$ on $S$ . As such it defines a rank-one torsion free sheaf on $A_{s-k}$ which we still  denote by $D'$.
From   (\ref{E:ineq1-2})  we get 
\be\label{contr_om(-D)}
h^0(A_{s-k}, \omega_{A_{s-k}}(-D'))\geq 2
\ee
Thus, by the Riemann-Roch theorem on $A_{s-k}$:
 \be\label{contr_D}
 h^0(A_{s-k}, \mathcal{O}_{A_{s-k}}(D'))=h^0({A_{s-k}}, \omega_{A_{s-k}}(-D'))+(d-l)-s+k+1\geq \frac{d+1}{2}-l+k+1
 \ee

 Therefore either  $h^0(A_{s-k}, \mathcal{O}_{A_{s-k}}(D'))=1$ and $d+1\leq 2$, implying that  $d\leq 1$, which is precisely what we aim at, or $ h^0(A_{s-k}, \mathcal{O}_{A_{s-k}}(D'))\geq 2$, which, together
 with (\ref{contr_om(-D)}) tells us that $D'$ contributes to the Clifford index of $A_{s-k}$.
 Let us see that this case can not occur.
By (\ref{contr_D}) we get 
 \be\label{cliff_D}
 \aligned
 \cliff{D'}&=d-l-2h^0(A_{s-k}, \mathcal{O}_{A_{s-k}}(D'))+2\\
 &\leq d-l-2\left( \frac{d+1}{2}-l+k+1\right)+2\leq -1+l-2k \leq -k-2\\
 \endaligned
 \ee
 and this is impossible because $A_{s-k}$ verifies the hypotheses of Theorem A in \cite{Franciosi-Tenni}.\end{proof}

As a Corollary, exactly as in  \cite{ABS1}, we get:

\begin{cor}\label{stab} For all $x \in S$ the locally free sheaf $E_x$ is stable on $C$. 
\end{cor}
\begin{proof}
Since Lemmas 5.2, 5.3 and 5.4
of \cite{ABS1} hold verbatim, we can apply Proposition 5.5  of \cite{ABS1}, and obtain the result.
\end{proof}

There is another important consequence of Lemma (\ref{stab}) to be used in the last section.
It is based on Mukai's Lemma 1 in \cite{Mukai2} whose statement we include for the convenience of the reader.

\begin{lem}\label{Mukai_lem1}{\rm (Mukai)}. Let $E$ be a rank two vector bundle on $C$ with canonical determinant, and let $\zeta$ be a line bundle on $C$. If $\zeta$ is generated by global sections, then we have
$$
\dim \Hom_{\mathcal{O}_C}(\zeta, E)\geq h^0(C, E)-\deg\zeta
$$
\end{lem}

\begin{cor}\label{stab2} Let
$$
0\longrightarrow L \longrightarrow E\longrightarrow K_CL^{-1}\longrightarrow 0
$$
be an extension on $C$ where $|L|$ is a is a base-point-free $g^1_{s+2}$. Then $E$ is stable.

\end{cor}
\begin{proof} This is proved exactly as in Remark 5.11 of \cite{ABS1} (with $E$ instead of $E_L$)  by using Mukai's Lemma \ref {Mukai_lem1}, and
Lemma  5.3 of \cite{ABS1},  which holds in our situation as well, while Lemma 4.3 of \cite{ABS1}
can be substituted by Lemma (\ref{stab}) above.
\end{proof}

The relevance of this Corollary is the following result asserting the existence on $C$ of  
base-point-free $g^1_{s+2}$'s.

\begin{lem}\label{g-1-s+1} Let $C$ be a smooth hyperplane section of an $(s+1)$-special Halphen surface.
Then there exists on $C$ a base-point-free $g^1_{s+2}$.
\end{lem}

\begin{proof} Here we can repeat, word by word, the proof of item iii) of Proposition 4.5 in \cite{ABS1}.
Recalling the notation introduced at the beginning of this section, the only result we need to check is that, also in our situation $h^0(C,\eta\xi^{-1})=1$. We look at the sequence
$$
0\longrightarrow \mcO_S(-2B) \longrightarrow \mcO_S(A-B)\longrightarrow \eta\xi^{-1}\longrightarrow 0
$$
and we readily conclude using  Proposition \ref {calcoli}, iii), iv) an Serre's duality.
\end{proof}
\begin{rem}\label{voisin-bund}
Let $C$ be a 
a smooth genus $g$ curve. To any pair $(v, L)$ where $v$ is a vector in the 
 cokernel of the Gauss-Wahl map (\ref{G-W}), and $L$ is a base-point-free pencil on $C$,
Voisin,  \cite{Voisin_W} 
associates  a rank-2 vector bundle $E_{L,v}$, often denoted simply by $E_L$. This vector bundle is an extension
$$
0\longrightarrow L \longrightarrow E_L\longrightarrow K_CL^{-1}\longrightarrow 0
$$
having the property that
\be\label{split-h0}
h^0(C,E_L)=h^0(C,L)+h^0(C,K_CL^{-1})
\ee
We call such a bundle a Voisin bundle.
Voisin interprets the vector $v$ as a ribbon in $\PP^g$ having the curve $C$ as hyperplane section.
Thanks to Theorem 7.1 in \cite{Wahl-square} and Theorem 3 in \cite{ABS2}, whenever $g\geq 11$, and
whenever the Clifford index of $C$ is greater or equal than 3, this ribbon can be integrated to a bona fide surface $X\subset \PP^g$ having isolated singularities and canonical hyperplane sections, among which $C$ itself.
When $X$ is a K3 surface,  $E_L=E_{L,X}$ is nothing but the restriction to $C$ of the Lazarsfeld-Mukai bundle
$\E_{L,X}$ on $X$ whose dual $\F_{L,X}$ is defined by
$$
0\longrightarrow \F_{L,X} \longrightarrow H^0(C, L)\otimes \mcO_S\longrightarrow L\longrightarrow 0
$$
When $g=2s+1$ and   $|L|$ is a base-point-free $g^1_{s+2}$ on $C$, then 
\be\label{split-h0v}
h^0(C,E_{L,X})=s+2
\ee

\end{rem}

\subsection{Brill Noether loci on the hyperplane section of an $(s+1)$-special Halphen surface.}

For any $x \in S$ we have produced a rank $2$,  locally free, stable, sheaf 
$$
E_x:={\wt{\E}{}_x}_{|C}
$$ 
on $C$
with determinant equal to $K_C$
and having $s+2$ linearly independent sections (Proposition \ref{sections}, and Corollary \ref{stab}).

Let 
$$
M_C(2,K_C,s+2)=\{v.b. \,\,E\,\,\text{on}\,\, C\,\,|\,\,\rk E=2\,,\,\,\det E=K_C\,,\,\,h^0(E)\geq s+2\}
$$
be the Brill-Noether locus of stable rank $2$ locally free sheaves on $C$ whose
determinant is $K_C$ with at least $s+2$ sections.

Via the universal family of extensions (\ref{ue}), 
we can  define  a morphism:
\be\label{sigma}
{\aligned
\sigma:T=S &\longrightarrow M_C(2,K_C,s+2)\\
& x\mapsto E_x
\endaligned}
\ee

By construction the map $\sigma$ contracts  the curve $J$ to a point, but we are now going
to show that this is the only fiber of $\sigma$ containing more than one point.

\begin{prop}\label{inj} The restriction of $\sigma$ to $S \setminus J$ is injective
\end{prop}

\begin{proof}
By construction, if $x\notin C$, the sheaf $E_x$ is an extension
$$
0\longrightarrow \xi \longrightarrow E_x\longrightarrow \eta\longrightarrow 0
$$
while, if $x \in C$, the sheaf $E_x$ is an extension
$$
0\longrightarrow \xi(x)\longrightarrow E_x\longrightarrow \eta(-x)\longrightarrow 0
$$
It will suffice to show that
$$
\dim \Hom_{\mathcal{O}_C}(\xi, E_x)=1.
$$
In order to do this, we consider the exact sequence on $S$:

$$
0\longrightarrow\wt\E_x(-B-C) \longrightarrow\wt\E_x(-B)
\longrightarrow E_x\otimes \xi^{-1}\longrightarrow 0
$$
Since $h^0(S, \wt \E_x(-B-C))=0$, and $h^0(S, \wt \E_x(-B))=1$, it will be enough to show that
$$
H^1(S, \wt\E_x(-B-C))=0
$$
From
\be\label{succ-c}
0\longrightarrow\mcO_S(-C) \longrightarrow\wt\E_x(-B-C)
\longrightarrow\I_x(-2B)\longrightarrow 0
\ee
we obtain $\chi(S, \wt \E_x(-B-C))=2s+1$. It will then be enough to show that $h^2(S, \wt\E_x(-B-C))=2s+1$.
From Serre duality, and from the identification $\wt \E_x(-C)=\wt\E^\vee_x$, we have that
$$
h^2(S, \wt\E_x(-B-C))= h^2(S, \wt\E^\vee_x(-B))=h^0(S, \wt\E_x(B-J))
$$
Let us consider the base-point-free-pencil trick on $S$ for $B$. From the exact sequence

$$
0\longrightarrow \mcO_S(-B) \longrightarrow H^0(S, B)\otimes \mcO_S
\longrightarrow\mcO_S( B)
\longrightarrow0
$$
we obtain
$$
0\longrightarrow \wt\E_x( -B-J) \longrightarrow H^0(S, B)\otimes \wt\E_x(-J)
\longrightarrow \wt\E_x(B-J)
\longrightarrow0
$$
From (\ref{succ-c}) we obtain $h^0(S, \wt\E_x( -B-J)) =0$.
From Lemma (\ref{h12}) we have $ h^1(S, \wt\E_x(-J)) =0$ and, since $\chi(S, \wt E_x(-J))=s$, we also get
$h^0(S, \wt\E_x( -J)) =s$. Then it is enough to show that
\be\label{bj)}
h^1(S, \wt\E_x( -B-J)) =1
\ee 
We will do this considering the exact sequence:
\be\label{-J}
0\longrightarrow \mcO_S(-J) \longrightarrow \wt\E_x( -B-J)
\longrightarrow \I_x(A-B-J)
\longrightarrow0
\ee

We first observe that, from Serre duality, from the identification $\wt \E_x(-C)=\wt\E^\vee_x$, and from the exact sequence
$$
0\longrightarrow \mcO_S(B-A) \longrightarrow \wt\E_x(-A)
\longrightarrow \I_x \longrightarrow0
$$ 
we get
\be\label{-A}
h^2(S, \wt\E_x(-B-J))=h^0(S,\wt\E_x(-A))=0
\ee
We then compute $h^i(S, \I_x(A-B-J))$ from
$$
0\longrightarrow \I_x(A-B-J) \longrightarrow \mcO_S(A-B-J)
\longrightarrow \mcO_S(A-B-J)_{|x}
\longrightarrow0
$$
We obtain
$$
0\longrightarrow \CC \longrightarrow H^1( S, \I_x(A-B-J)
\longrightarrow H^1(S, \mcO_S( A-B-J))=H^1(S, \mcO_S( B-A))=\CC
\longrightarrow0
$$
and $ H^2(S, \I_x(A-B-J))=0$
Consider then (\ref{-J}). We have obtained
$$
0\longrightarrow H^1(S, \wt\E_x(-B-J) )\longrightarrow \CC^2
\longrightarrow \CC\longrightarrow 0
$$
which gives $h^1(S,\wt\E_x(-B-J) )=1$,

\end{proof}

\begin{prop}\label{surj}
$\sigma$ is surjective.
\end{prop}
 \begin{proof}

Since $h^0(C, E)\geq s+2$, by the preceding lemma, there must be an exact sequence
$$
0\to \xi(D)\to E\to\eta(-D)\to 0
$$
for some effective divisor $D$ of degree $d$ on $C$. Since $E$ is stable we must have
$$
 \deg(\xi(D))=s+1+d\leq \deg(E)/2=2s\,,
 $$
i.e. $d \le s-1$. We have
$$
s+2\leq h^0(\xi(D))+ h^0(\eta(-D))=2h^0(\eta(-D))-s+1+d
$$
so that
$$
h^0(\eta(-D))\geq s-\frac{d-1}{2}\geq\frac{s+2}{2}\geq 3
$$
since $s\geq5$.
We can the apply  Lemma (\ref{stab}), and deduce that $d \le 1$. Two cases can occur. Either:
\[
0 \longrightarrow \xi(p) \longrightarrow E \longrightarrow \eta(-p) \longrightarrow 0
\]
or
\[
0 \longrightarrow \xi \longrightarrow E \longrightarrow \eta \longrightarrow 0
\]
In the first case $E \cong E_p$ because the extension does not split and is unique. In the second case the coboundary 
$$
H^0(C, \eta)\to H^1(C, \xi)
$$ 
has rank one and then it corresponds to a point of the quadric hull $\phi_{|\eta|}(C)$. From Proposition (\ref{qhull}) the quadratic hull of $\phi_{|\eta|}(C)$ is exactly $\ov S$. We thus find a point 
$x\in S$
such that $H^0(S,\ I_x(A))= \mathrm{Im}[H^0(S, E)\longrightarrow H^0(C, \eta)=H^0(S, A)]$. Again by  uniqueness, we have $E=E_x=\wt \E_{x_{|C}}$. 
\end {proof}
\vskip 0.5 cm

At this stage we have a well defined morphism
$$
\sigma: S\longrightarrow M_C(2, K_C, s)
$$
such that 

{\bf a)}  $\sigma(J)=[E]$, where $E=E_x$, and $x$ is any point in $J$.

{\bf b)} $\sigma: S\setminus J\longrightarrow M_C(2, K_C, s)\setminus \{[E]\} $ is bijective

In particular there is an induced, bijective, morphism

\be\label {ovs}
\ov\sigma: \ov S\longrightarrow M_C(2, K_C, s)\
\ee

In order to prove that $\ov\sigma$ is an isomorphism, in the next sub-section we will prove 

\begin{prop}\label{smoothm}$M_C(2, K_C, s)\setminus \{[E]\}$ is smooth.
\end{prop}

\begin{prop}\label{norm} $M_C(2, K_C, s)$ is normal. 
\end{prop}

Since $\ov S$ is normal the consequence will be:

\begin{cor}\label{iso} $\ov \sigma$ is an isomorphism.
\end{cor}

\vskip 0.5 cm 
\subsection{The Petri map for some bundles on $C$.}\label{petri-map-s}
Let $(S,C)$ be as in the previous section.  We denote 
by 
$M_C(2, K_C)$ the moduli space of rank two, semistable vector bundles on $C$ with determinant equal to $K_C$, containing the Brill-Noether locus
$$
M_C(2, K_C, s)= \{[F]\in M_C(2, K_C)\,\,\,|\,\,\, h^0(C, F)\geq s+2\}
$$

A point  $[F]\in M_C(2, K_C)$,   corresponding to a stable bundle $F$, is a  smooth point of  $M_C(2, K_C)$, and
$$
T_{[F]}(M_C(2, K_C))=H^0(S^2F)^\vee\cong\CC^{3g-3}\,,
$$
In particular, since $\chi(C,S^2F)=3g-3$, this shows that if $F$ is any stable rank $2$ locally free sheaf of determinant $K_C$, then
\be\label{h^1(S^"F)}
h^1(S^2F)=0
\ee

It is well known that the Zariski tangent space to the Brill-Noether locus $M_C(2, K_C, s)$ at a point $[F]$  can be expressed in terms
of the ``Petri'' map
\be\label{petri-map} \mu: S^2H^0(C, F)\longrightarrow H^0(C, S^2F)
\ee
Indeed
\be\label{tang-BN}
T_{[F]}(M_C(2, K_C, s))=\operatorname{Im}(\mu)^\perp
\ee
We will compute the tangent space at $E_x$ for $x \notin J$ by 
considering the map $$S^2H^0(S, \wt \E_x) \longrightarrow H^0(S, S^2 \wt \E_x).$$

 From the exact sequence  (\ref{ext-abx}) we deduce the following exact sequence

\be\label{ex_s2ab}
0\to \wt \E_x(B) \to S^2\wt\E_x\to \I_x^2(2A)\to 0
\ee

\begin{prop}\label{prelPetri}For $x\notin J$, we have:
\begin{itemize}
	\item[(i)]   $h^0(S, S^2 \wt \E_x(-C))=0$, and $h^1(S, S^2 \wt \E_x(-C))=5$, 
	
	\item[(ii)]  $h^0(S, \wt \E_x(B))=2s+3$,  $h^1(S, \wt \E_x(B))=2$, and $h^2(S, \wt \E_x(B))=0$,
	
	\item[(iii)] $h^1(S, S^2 \wt \E_x)=3$.
	\end{itemize}
\end{prop}
\begin{proof}
Consider the exact sequence
\be
0\to \wt \E_x(-A)\to S^2\wt \E_x(-C)\to \I_x^2(A-B) \to 0
\ee
From Proposition (\ref{A})  we get  $h^0(S, \I_x^2(A-B))=0$, and from (\ref{-A}) we get
$h^0(S, \E_x(-A))=0$, so that $h^0(S, S^2 \wt \E_x(-C))=0$. From Proposition \ref{calcoli} iv)
we get 
$$
\chi(S, \wt \E_x(-A))=\chi(S,\mcO_S( B-A))+\chi(S, \I_x)=-1
$$
Moreover,  from Proposition \ref{calcoli} iv), and from the exact sequences
$$
0\to \I_x(A-B)\to \mcO_S(A-B)\to \mcO_S(A-B)_{|x}\to 0 
$$
and
$$
0\to \I_x^2(A-B)\to \I_x(A-B)\to \mcO_S(A-B)_{|x}\otimes \I_x/\I_x^2\to 0
$$
we get $\chi(S,  \I_x^2(A-B))=\chi(S, \mcO_S(A-B))-3=-4$.
It follows that 
$$
\chi(S, S^2 \wt \E_x(-C))=-5
$$
In order to prove i) it suffices to show that  $h^2(S, S^2 \wt \E_x(-C))=0$.
This follows at once from the two  equalities:  $h^2(S, \I_x^2(A-B)) =h^2(S,\mcO_S( A-B))=0$,  and $h^2(S, \E_x(-A))=h^2(S, \mcO_S(B-A))+h^2(S, \I_x)=0$.

Let us consider the cohomology of $\wt\E_x(B)$.
From the exact sequence
\be
0\to \mcO_S(2B)\to \wt \E_x(B)\to\I_x(C) \to 0
\ee
and from Proposition \ref{calcoli} ii), we get that $\chi(S, \wt \E_x(B))=2s+1$ and $h^2(S, \wt \E_x(B))=0$.
To complete item ii) it suffices  to  prove that $h^1(S, \wt \E_x(B))=2$.
In Lemma   (\ref{h12}) we showed that, when $x\notin J$, then $h^0(S, \wt \E_x(-J))=s$ and $h^1(S, \wt \E_x(-J))=0$.
From (\ref{-J}) we have $h^0(S, \wt \E_x(-B-J))=0$ and $h^1(S, \wt \E_x(-B-J))=1$. Since $\chi(S, \wt \E_x(-B-J))=-1$, this gives $h^2(S, \wt \E_x(-B-J))=0$.

We apply all of this to the exact sequence:
\be
0\to \wt \E_x(-B-J)\to \wt \E_x(-J)\to \wt \E_x(-J)|_B\to 0
\ee
and we  deduce that $h^1(B, \wt \E_x(-J)|_B )=h^1(B, \wt \E_x(B)|_B )=0$ (recall that $B$ and $J$ are trivial on $B$).
From the sequence
\be
0\to \wt \E_x\to \wt \E_x(B)\to \wt \E_x(B)|_B\to 0
\ee
we get $h^1(S, \wt \E_x(B)) \leq 2$. 
In order to prove that equality holds, we consider 
the base-point-free-pencil trick sequence for $B$ twisted by $\wt \E_x$:

\be
0\to \wt \E_x(-B)\to H^0(S, \mcO_S(B)) \otimes  \wt \E_x\to \wt \E_x(B) \to 0
\ee
Since $h^2(S, \wt \E_x(-B))=0, h^1(S, \wt \E_x(-B))=2$,
we get $h^1(S, \wt \E_x(B))= 2$. This proves ii).

As for iii), using ii),  the sequence \ref{ex_s2ab}, and the fact that $h^1(S, \I_x^2(A))=h^1(S, \mcO_S(2A))=1$,  we get $h^1(S, S^2 \wt \E_x) \leq 3$. 
In order to prove equality we consider the sequence: 
\be
0\to S^2 \wt \E_x(-J)\to S^2 \wt \E_x\to S^2 \wt \E_x|_J\to 0
\ee
Since ${\wt \E_x}|_J \cong \mcO_J^2$ and $h^2(S, S^2 \wt \E_x(-J))=h^0(S, S^2\wt \E_x^{\vee})=0$,
we have $h^1(S, S^2 \wt \E_x) \geq  h^1(S, {S^2 \wt \E_x}|_J ) =3$.
\end{proof}

We consider again the two sequences

\be\label{ex_s2ab-2}
0\to \wt \E_x(B) \to S^2\wt\E_x\to \I_x^2(2A)\to 0
\ee

\be\label{exb}
0\to \mcO_S(2B))\to\wt\E_x(B)\to (\I_x(A+B)\to 0
\ee

and  the two exact sequences (the first of which defines the vector space  $U$)

\be\label{ex_S2ab}
0\to U\to S^2H^0(\wt \E_x)\to S^2H^0(I_xA)\to 0
\ee

\be\label{U}
0\to S^2H^0(S,\mcO_S(B))\to U\to H^0(S, \mcO_S(B))\otimes H^0(\I_x(A))\to 0
\ee

\begin{prop}\label{s2} The map
 $c: S^2H^0(S, \wt \E_x)\longrightarrow H^0(S, S^2\wt \E_x)$ is surjective.
\end{prop}
\begin{proof}
From sequences (\ref{exb}), (\ref{ex_S2ab}) and (\ref{U}) we get a diagram:
$$
\xymatrix{
0\ar[r]&U\ar[r]\ar[d]_u&S^2H^0(S,\wt \E_x)\ar[d]^c\ar[r]^l&S^2H^0(S, \I_x(A))\ar[d]^F\ar[r]&0\\
0\ar[r]&H^0(S, \wt \E_x(B))\ar[r]&H^0(S, S^2\wt \E_x)\ar[r]^m&H^0(\I_x^2(2A))\ar[r]&0\\
}
$$
where $m$ is surjective from Proposition \ref{prelPetri} ii),  iii),  and the fact that $h^1(\I_x^2(2A))=1$, 
(as follows from Proposition \ref{A} iii)).
We claim that $u$ is an isomorphism. 
Consider the diagram
$$
\xymatrix{
0\ar[r]&S^2H^0(S, \mcO_S(B))\ar[r]\ar[d]&U\ar[d]_u\ar[r]&H^0(S, \mcO_S(B))\otimes H^0(S,\I_x(A))\ar[d]\ar[r]&0\\
0\ar[r]&H^0(S,\mcO_S(2B))\ar[r]&H^0(S, \wt \E_x(B))\ar[r]^-v&H^0(S, \I_x(A+B))
}
$$
Since, from base-point-free-pencil trick,  $S^2H^0(\mcO_S(B))\to H^0(\mcO_S(2B))$ is an isomorphism and
$$
H^0(S, \mcO_S(B))\otimes H^0(\I_x(A))\to H^0(\I_x(A+B))
$$ 
is injective, the claim follows, via a dimension count,  from Proposition \ref{prelPetri} ii).
From Lemma (5.13) of \cite{ABS1}, which can be applied verbatim in our situation, also $F$ is surjective, so that $c$ is also surjective.

\end{proof}

As a corollary we have:

 \begin{cor}\label{dim-tan}
For every $x\notin J$ 
$$
\dim T_{[E_x]}(M_C(2,K_C,s))=2
$$
(This proves Proposition \ref{smoothm}.)
\end{cor}
\begin{proof}
 Consider the commutative diagram of maps:
\[
\xymatrix{
S^2H^0(C, E_x) \ar[r]^-a & H^0(C, S^2E_x) \\
S^2H^0(S, \wt \E_x) \ar[u]^-b \ar[r]_-c & H^0(S, S^2 \wt \E_x) \ar[u]_-d
\
}
\]
Proposition \ref{prelPetri} i) implies that $d$ is injective, while   Proposition (\ref{s2}) tells that $c$ is surjective. Therefore 
\[
\mathrm{corank}(a) \le \mathrm{corank}(d) 
\]
From the sequence
$$
0 \to S^2 \wt \E_x(-C) \to S^2 \wt \E_x \to S^2E_x \to 0
$$
from  (\ref{h^1(S^"F)}),  and from  Proposition \ref{prelPetri} ii) and iii) 
we get that $\mathrm{corank}(a) \le 2$. Since coker$(a)^\perp = T_{[E_x]}M_C(2,K_C,s)$, we conclude  that 
$\mathrm{dim}[T_{[E_x]}M_C(2,K_C,s)]\le 2$.   But from Proposition (\ref{surj})  it follows that $T_{[E_x]}M_C(2,K_C,s)$ has dimension at least $2$ at $[E_x]$, and this proves the result. 
\end{proof}

\proof (of Proposition \ref{smoothm}). The fact that $\sigma: S\setminus J\to M_C(2,K_C, s)\setminus\{[E]\}$ is bijective tells us that $\dim M_C(2,K_C, s)=2$.
 Proposition  \ref{smoothm} follows then from Corollary \ref{dim-tan}.
\endproof

The surface $M_C(2,K_C, s)$ has an isolated singularity at the point $[E]$, and is otherwise smooth. In order to show that 
$M_C(2,K_C, s)$ is normal it suffices to prove the following Lemma
\begin{lem}\label{tane}
$$
\dim T_{[E]}M_C(2,K_C, s)=3
$$
\end{lem}

\begin{proof}
The proof  will follow the path that led to the proof of  Corollary \ref{dim-tan}. The ingredients will be essentially the same 
but the various computations will  be drastically different. We start with the commutative diagram:
\be\label{comm-e}
\xymatrix{
S^2H^0(C, E) \ar[r]^-a & H^0(C, S^2E) \\
S^2H^0(S,  \E) \ar[u]^-b \ar[r]_-c & H^0(S, S^2 \E) \ar[u]_-d
\
}
\ee
where 
$$
0\longrightarrow\mcO_S(B) \longrightarrow\E\longrightarrow \mcO_S(A)\longrightarrow0
$$
is the unique non split extension of  $\mcO_S(A)$  by $\mcO_S(B)$, and $E=\E_{|C}$. We must prove that

\be\label{crk3}
\operatorname{corank}a=3
\ee

\begin{lem}\label{s2e} $h^0(S^2\E(-C))=h^2(S^2\E(-C))=0$, $h^1((S^2\E(-C))=h^1((S^2\E)=1$
\end{lem}
\begin{proof} From the basic sequence $0\to \mcO_S(B)\to\E\to \mcO_S(A)\to 0$, we deduce the exact sequence
$0\to \mcO_S(B-A)\to\E(-A)\to \mcO_S\to 0$. The coboundary of this sequence is given by the extension class
which  is not zero. It follows that $h^0(\E(-A))=h^2(\E(-A))=0$. Now look at the exact sequence
$$
 0\longrightarrow \E(-A)\longrightarrow S^2\E(-C)\longrightarrow \mcO_S(A-B)\longrightarrow 0
$$
By what we just remarked, we get that $h^0(S^2\E(-C))=h^2(S^2\E(-C))=0$. On the other hand
 $\chi(S^2\E(-C))=-1$, so that $h^1((S^2\E(-C))=1$. The exact sequence
\be\label{restr-s2e}
 0\longrightarrow S^2\E(-C)\longrightarrow S^2\E\longrightarrow S^2E\longrightarrow 0
\ee
and the fact that $h^1(S^2E)=0$ (Corollary \ref{stab}), we deduce that $h^1(S^2\E)\leq 1$.
Since $\mcO_J(A)=\mcO_J(B)=\mcO_J$ we have an exact sequence

$$
 0\longrightarrow \mcO_J\longrightarrow \E_{|J}\longrightarrow \mcO_J\longrightarrow 0
$$
from which we deduce the exact sequence
$$
 0\longrightarrow \E_{|J} \longrightarrow S^2\E_{|J}\longrightarrow \mcO_{J}\longrightarrow 0
$$
It follows that $h^1(S^2\E_{|J})\neq0$. Since $h^2(S^2\E(-J))=h^0(S^2\E(-2C))=0$ we get that  $h^1(S^2\E)\neq0$.
This proves the lemma.

\end{proof}

\begin{lem}\label{bd}Both $b$ and $d$, in the above diagram are isomorphisms.
\end{lem}
\begin{proof}The statement about $b$ follows from (\ref{he1}). The statement about $d$ follows from the exact sequence (\ref{restr-s2e})
and Lemma \ref{s2e}.
\end{proof}

To prove  Lemma \ref{tane}, we are now reduced to prove
\be\label{crk3}
\operatorname{corank}c=3
\ee 
Consider diagram \ref{tors-free}. The first row is exact at the level of global sections. Also since $A\cdot J=0$, we must have 
$H^0(S, \I_x(A))=H^0(S, \mcO_S(A-J))$, whenever $x\in J$.
It follows that  the homomorphism $H^0(S, \E)\to  H^0(S, \mcO_S(A))$ factors through a surjective homomorphism onto $H^0(S, \mcO_S(A-J))$,
and we have the exact sequence 
$$
0\longrightarrow H^0(S, \mcO_S(B))\longrightarrow H^0(S, \E)\longrightarrow H^0(S, \mcO_S(A-J))\longrightarrow 0
$$
giving rise to two exact sequences
$$
0\longrightarrow U\longrightarrow S^2H^0(S, \E)\longrightarrow S^2H^0(S, \mcO_S(A-J))\longrightarrow 0
$$
$$
0\longrightarrow S^2H^0(S, \mcO_S(B)) \longrightarrow U\longrightarrow  H^0(S, \mcO_S(B))\otimes H^0(S, \mcO_S(A-J))\longrightarrow 0
$$
and a diagram
\be\label{diagrammone}
\xymatrix{
0\ar[r]&U\ar[r]\ar[d]_u&S^2H^0( \E)\ar[d]^c\ar[r]^{l\qquad}&S^2H^0(\mcO_S(A-J))\ar[d]^F\ar[r]&0\\
0\ar[r]&H^0(  \E(B))\ar[r]&H^0(S^2 \E)\ar[r]^m&H^0(\mcO_S(2A))\ar[r]^\alpha&H^1(  \E(B))\cong\CC\ar[r]&0\\
}
\ee
To explain the homorphism $\alpha$, observe that, a priori, we have an exact sequence
$$
H^0(\mcO_S(2A))\overset \alpha\longrightarrow H^1( \E(B))\longrightarrow H^1( S^2\E) \longrightarrow H^1( \mcO_S(2A) \longrightarrow H^2(  \E(B))
$$
We already know that $h^1( S^2\E)=h^1(  \mcO_S(2A)=1$, therefore it suffices to prove the following lemma.
\begin{lem}\label{s2eb}$h^1( \E(B))=1$, $h^2( \E(B))=0$
\end{lem}
\begin{proof}
Look at the exact sequence coming from the base-point-free-pencil-trick
$$
0\longrightarrow \E(-B)\longrightarrow H^0(S,  \mcO_S(B))\otimes\E\longrightarrow  \E(B) \longrightarrow 0
$$
The cohomology of $\E(-B)$ is readily computed via $0\to\mcO_S\to\E(-B)\to \mcO_S(A-B)\to 0$
and one gets $h^0(\E(-B))=h^2(\E(-B))=0$, $h^1(\E(-B))=1$. The lemma follows from this.
\end{proof}

We can now go back to diagram (\ref{diagrammone}).
\begin{lem}\label{s2eb}The homomorphism $u$ is surjective.
\end{lem}
\begin{proof} Since $h^1(S,  \mcO_S(2B))=2$, $h^1(S,  \E(B))=h^1(S,  \mcO_S(C))=1$, we have the following diagram with exact rows.
\be\label{diagrammone2}
\xymatrix{
0\ar[r]&S^2H^0(\mcO_S(B))\ar[r]\ar[d]^\cong&U\ar[d]^u\ar[r]^{w\qquad\qquad\qquad}&H^0(\mcO_S(B))\otimes H^0( \mcO_S(A-J))\ar[d]^\gamma\ar[r]&0\\
0\ar[r]&H^0(\mcO_S(2B))\ar[r]&H^0(  \E(B))\ar[r]^r&H^0(\mcO_S(C))\ar[r]^\beta&\CC^2\ar[r]&0\\
}
\ee
The kernel of $\gamma$ is $H^0(\mcO_S(A-B-J)=0$. The image of $r$ has dimension 
$$
2s=\dim H^0(\mcO_S(B)\otimes H^0(\mcO_S(A-J))\,.
$$
\end{proof}
\vskip 0.2 cm
Let us go back to diagram (\ref{diagrammone}). Let $W=\operatorname{Im}(m)$. We have
$$
\operatorname{Im}(F)\subset W\subset H^0(\mcO_S(2A))\,,\qquad \codim_{H^0(\mcO_S(2A))} W=1
$$
\be\label{corank-f}
\aligned
&\dim(\coker(c)=\dim\coker \{F: S^2H^0(\mcO_S(A-J))\to W\}\\
&=\dim\coker \{F: S^2H^0(\mcO_S(A-J))\to H^0(\mcO_S(2A))\}-1\\
\endaligned 
\ee
The homomorphism $F: S^2H^0(\mcO_S(A-J))\to H^0(\mcO_S(2A))$ factors through $H^0(\mcO_S(2A-2J))$.
Let us  study the linear system $|A-J|$. We have
$$
A-J=(s-1)J+ (s-1)E_{10}+E_9
$$

Since $(A-J)\cdot E_{10}=0$, we may as well work in $S'$. There we consider the divisor $A'=sJ'+E_9$, so that
\be\label{a-j}
A'-J'=(s-1)J'+E_9
\ee
whose proper transform under $S\to S'$ is $A-J$. Let $\tau: S_1\to S'$ be the blow up of  $S'$ at $p_{11}=p_{10}(s-1)$.
Let $J_1$ be the proper transform of $J'$ and $E_{11}$ the preimage of $p_{11}$ under $\tau$.
The proper transform of (\ref{a-j}) in $S_1$ is 
$$
A_1=(s-1)J_1+ (s-2)E_{11}+E_9=\tau^*(A'-J')-E_{11}
$$
By construction
$$
H^0(\mcO_{S_1}(A_1))=H^0(\mcO_{S}(A-J))=H^0(\mcO_{S'}(A'-J'))
$$
We see that $A_1$ is a Brill-Noether-Petri du Val curve of genus $s-1$; since $s-1\geq5$,
its canonical image is projectively normal.
As a consequence
the homomorphism  $\lambda$
$$
S^2H^0(\mcO_{S_1}(A_1))\overset\lambda\longrightarrow H^0(\mcO_{S_1}(2A_1))
$$
is surjective.
On the other hand the sections of $H^0(\mcO_{S}(A-J))$ vanish on $p_{11}$ so that we may identify $H^0(\mcO_{S_1}(2A_1))$ with $\operatorname{Im}(F)$.
Since $A_1$ is a du Val curve on $S_1$ we have $h^0(\mcO_{S_1}(2A_1))=4s-6$ while $h^0(\mcO_{S}(2A))=4s-2$ (we are using Proposition \ref{A} for both $A$ and $A_1$).
From (\ref{corank-f}), it follows that $\dim(\coker (c))=3$. This finishes the proof of Lemma \ref{tane}.

\end{proof}
\begin{proof} (of Proposition \ref{norm}, and Corollary \ref{iso}) This follows from the very well known fact that an isolated surface singularity 
with embedding dimension equal to 3 is normal.
\end{proof}

\section{On the corank of the Gauss-Wahl map of a general du Val curve.}

We will prove our main Theorem \ref{main-theorem} by showing, that in fact the corank-one  property holds for du Val curves which are hyperplane sections of Halphen surfaces
of index (s+1).

We recall that surfaces with canonical hyperplane sections which are not cones, or smooth $K3$ surfaces, are classified by Epema (see e.g.\cite{ABS2}, Section 9).

 If $\ov S$ is such a surface, we have a diagram

\be\label{E:fakediag}
\xymatrix{
S\ar[d]_p\ar[r]^--q&\ov S\subset\PP^g\\
S_0
}
\ee

where $q$ is the minimal resolution and $S_0$ is a minimal model of $S$. The minimal model $S_0$ can be either a ruled surface over a curve $\Gamma$ or $\PP^2$.
Notice that, since $S_0$ is ruled, $S$ has many minimal models which are connected by birational transformations whose graph is dominated by $S$.
For instance, if $S_0$ is a rational surface, we can always assume that $S_0=\PP^2$ (see for instance Section 11.3 of \cite{ABS2}).
Suppose $C \subset S$ is a du Val curve, so that $S_0=\PP^2$, and 
$$
C=3g\ell-gE_1-\cdots-gE_8-(g-1)E_9-E_{10}
$$
we can change the plane model $C_0$ by any quadratic transformation centered at any three points among $p_1, \ldots, p_{10}$. We will  call this transformed curve a birational du Val curve.

Before proving the main Theorem of this section, we recall Theorem 6.1 in \cite {ABS1}:

\begin{theorem}\label{Tmain}
Let $(S,C)$ be a general polarised $K3$ surface with $\Pic(S) \cong \ZZ\cdot[C]$ and $g=2s+1$.
Let $v= (2, [C],s)$ so that $M_v(S)$ is a smooth polarised $K3$ surface. 
There is a unique, generically smooth, 2-dimensional irreducible component  $V_C(2,K_C,s)$ of $M_C(2,K_C,s)$, containing the Voisin bundles $E_L$, with $L\in W^1_{s+2}(C)$,
such that 
  $\sigma$ induces   an isomorphism of $M_v(S)$ onto $V_C(2,K_C,s)_{red}$. In particular   $V_C(2,K_C,s)_{red}$ is a  K3 surface.
\end{theorem}

\begin{theorem}\label{Main-s+1} Let $C\subset\PP^{g-1}$ be a du Val curve of genus $g=2s+1>11$ which is the hyperplane section
of a polarised Halphen surface of index $ (s+1)$. Then the corank of the Gauss-Wahl map for $C$ is equal to $1$.
\end{theorem}

 We will start by proving  an intermediate result which, in fact, catches  the geometric significance of
 Theorem \ref{Main-s+1}.

\begin{theorem}\label{Main-s+1-smooth} Let $C\subset\PP^{g-1}$ be a du Val curve of genus $g=2s+1>11$ which is the hyperplane section
of a polarised Halphen surface of index $ (s+1)$. Then $C$ is not a hyperplane section of a smooth K3 surface in $\PP^g$.
\end{theorem}

\begin{proof} Suppose, by contradiction, that $C$ is a hyperplane section of a smooth $K3$ surface $X\subset \PP^g$.
We can choose   families
$$
\xymatrix{\qquad\qquad\X\subset \PP^g\times T\ar[d]^h\\(T,t_0)}
\,\qquad
\xymatrix{\qquad\qquad\C\subset \PP^{g-1}\times T\ar[d]^k\\(T,t_0)}
$$
parametrised by a smooth pointed curve $(T,t_0)$ having the following properties.
A fiber $X_t$ of $h$ is $K3$ surface in $\PP^g$. A fibre  $C_t$ of  $k$ is a hyperplane section of $X_t$.
Also $X_{t_0}=X$, and $C_{t_0}=C$. 
For $t$ belonging to a dense open set $A\subset T$,  
the Picard group of $X_t$ is of rank  $1$ and  generated by $[C_t]$. As in Theorem \ref{Tmain},
if $v_t= (2, [C_t],s)$, there is an isomorphism 
\be\label{iso-m-v}
M_{v_t}(X_t) \cong  V_C(2,K_{C_t},s)_{red}\,,\qquad t\in A.
\ee
Moreover,  for $t\in A$, the moduli space $M_{v_t}(X_t)$ is a smooth $K3$ surface. Following  \cite{Huybrechts-Lehn}, Theorem 4.3.7,
we can construct relative families of moduli spaces

$$
\sigma: \M \longrightarrow (T,t_0)\,,\qquad \tau: \V \longrightarrow (T,t_0)\,,
$$
where 
$$
\M_t=M_{v_t}(X_t)\,,\qquad \V_t=V(2, K_{C_t},s)
$$
Since for  $t\in A$,  the fiber $\M_t$ is a smooth $K3$ surface, we may find
an analytic neighbourhood $\Delta\subset T$ of $t_0$ such  that, for $t\neq t _0$ the fiber $ \M_t$ is a smooth $K3$ surface. By virtue of (\ref{iso-m-v}), we may also assume that for $t\in\Delta\setminus\{t _0\}$ the fiber $ \V_t$ is a smooth $K3$ surface.
By Corollary \ref{iso}, we know that $M(2, K_{C},s)=\ov S$, in particular
$$
V_{C_{t_0}}(2,K_{C_{t_0}},s)_{red}=V_{C_{t_0}}(2,K_{C_{t_0}},s)=M(2, K_{C},s)=\ov S
$$
and as a consequence $V_{C_{t}}(2,K_{C_{t}},s)_{red}=V_{C_{t}}(2,K_{C_{t}},s)$, for  $t\in \Delta$.
\begin{claim}
 Set $v=v_{t_0}=(2, [C],s)$. Then $M_{v_{t_0}}(X_{t_0})=M_v(X)$ is a K3 surface with isolated singularities of type $A_1$.
\end{claim}
{\it Proof of the Claim}.
Let us first produce a smooth point in $M_v(X)$. The K3 surface $X$
contains the curve $C$, on which, by Lemma \ref{g-1-s+1}, there exists a base-point-free $g^1_{s+2}=|L|$.
Starting from $X$, and recalling Remark \ref{voisin-bund}, we can consider the rank two vector bundles 
$\E_{L,X}$, and $E_{L,X}$ on $X$ and $C$ respectively. By Lemma \ref{stab2}, $E_{L,X}$ is stable.
Repeating word by word the argument at the end of  the proof of Proposition 5.5 in \cite {ABS1}, we get that
$\E_{L,X}$ is stable as well and represents a smooth point of $M_v(X)$. Since  $v^2=0$, we can conclude that
$M_v(X)$ is a surface with isolated singularities. A singular point $p$ of $M_v(X)$ corresponds to a polistable sheaf
$F_1\oplus F_2$ with $F_1\not\cong F_2$, as $v$ is primitive. Set $v_i=v(F_i)$,  we necessarily have $v_i^2=-2$
and $v_1\cdot v_2=2$. Following \cite{Arbarello-Sacca12} the quadratic cone at $p$ to $M_v(X)$
is given by an equation of type $z^2=xy$.
This proves the Claim.
\vskip 0.2 cm
Over $\Delta\setminus \{t_0\}$ we then have two families of smooth $K3$ surfaces which, by  Theorem (\ref{Tmain}), are fiberwise isomorphic
over the dense subset $A\cap(\Delta\setminus \{t_0\})$. By the ``Principal Lemma'' of Burns-Rapoport  (see e.g.
\cite{Beauville-Torelli-K3} Expos\'e IX) the two families are isomorphic over $\Delta\setminus \{t_0\}$. But now, 
$\sigma$ is a degenerating family of  $K3$ surfaces whose semistable model 
has a smooth K3 surface in the central fiber, while $\tau$ is a degenerating family of  $K3$ surfaces  whose semistable model 
has, as central fiber, the union of two smooth rational surfaces meeting on an elliptic curve.
By Kulikov's theorem \cite{Kulikov} (see also \cite{Persson-Pinkham},  \cite{Morrison-C-S-exact}) the two families have different monodromy
and can not be isomorphic over $\Delta\setminus \{t_0\}$. A contradiction.
\end{proof}

\vskip 0.2 cm
\begin{proof} (of Theorem \ref{Main-s+1}). Let $U$ be a $2$-dimensional subspace of the cokernel of the Gauss-Wahl map $\nu$ containing the point  which correspond to $\ov S$. From Theorems  3 in \cite{ABS2}, and 7.1 in \cite{Wahl-square}
we have a flat family 
$$
f:  \X \longrightarrow (T, t_0)\,,\qquad X_{t}=f^{-1}(t)\,,\qquad X_{t_0}=\ov S
$$
whose fibers are surfaces having $C$ as hyperplane section and where $T$ is a finite cover of $\PP (U)$.  In view of Theorem \ref{Main-s+1},
we may assume that, for general $t\in T$, the surface  $X_{t}$ is a singular surface with canonical hyperplane sections. Since it degenerates to $\ov S$, the surface  $ X_{t}$ has only one elliptic singularity.
The family $f$ has a section given by the singular point of each fiber. Since the singularity of the central fiber is resolved by the blow-up of the elliptic
singularity, the blow-up of such a section is a family
$$
f': \X' \longrightarrow T
$$
whose general fiber is a smooth surface degenerating to $S$. Up to removing finitely many points from $T$ the exceptional divisor $\J$ of such a blow-up
is a Cartier divisor; then we have $\J^2 \cdot f'^{-1}(t)=-1$.
Since $S$ is rational, all the fibers of $f'$ are rational and  we  can consider them as blow-ups of the plane.
Such a datum  defines a family of $g^2_d$'s on $C$ which are birationally ample.
Since  $\J^2 \cdot f'^{-1}(t)=-1$, this defines a family of plane curves of degree $d$ which have the same geometric genus and
at most $10$ singular points: in fact from Epema's classification the singular locus of the plane model of $C$ for $t\in T$ lies on the plane cubic
$J_{t}=\J_{|f'^{-1}(t)}$ and $J_{t}^2=-1$. This means that multiplicities of the singular points must be constant on the family.
Let us write the central fiber of such a family as
$$
C=d\ell'-m_1E'_1-\cdots-m_{10}E'_{10}
$$ 

By hypothesis, there exists a birational transformation $\xymatrix{\PP^2 \ar@{<-->}[r]^\phi& \PP^2}$ whose graph is dominated by $\ov S$
which transforms the central fiber into a curve 
$$
C=3g\ell-gE_1-\cdots-gE_8-(g-1)E_9-E_{10}
$$ 
Such a birational transformation is induced by a  birationally ample linear system of plane curves 
$$
F=r\ell'-n_1E'_1-\cdots-n_{10}E'_{10}\,.
$$
For $t\in T$ consider the linear system 
$$
F_t=r\ell'_t-n_1E'_{1,t}-\cdots-n_{10}E'_{10, t}\,.
$$
We have $F_t^2=1$ and $|F_t|$ is base-point-free showing that 
 $C_t$ is a birational du Val curve.
Since all surfaces $X_{t}$ of the family have $C$ as hyperplane section and since $C$ is not Brill-Noether general, from
Theorem 3.2 of \cite{ABFS}, each surface $X_{t}$ must be Halphen of some index $m \leq 2s+1$. But, since the central fiber
is Halphen of index $s+1$, the integer $m$ must be greater or equal than $s+1$, but also  a multiple of $s+1$.
This means that all fibers of the family $f$ are polarised Halphen surfaces of index $(s+1)$. Then Corollary (\ref{iso}) applies  to them
and we get that such surfaces are all isomorphic as polarised Halphen surfaces of genus $2s+1$. Since these surfaces have a discrete group
of automorphisms (see \cite{Cantat-Dolgachev}) they must be generically obtained one from the other by a projectivity and this is a contradiction.
\end{proof}

\begin{proof} (of Theorem \ref{main-theorem}). The Theorem follows from Theorem \ref{Main-s+1},
by semicontinuity and from the irreducibility of the space of du Val curves.
\end{proof}
\begin{proof} (of Corollary \ref{main-cor}). This follows immediately from Theorem \ref{main-theorem} and from Corollary 3.12 in \cite{Wahl-sections}
with $E=C$ and $X=\ov S$. In that Corollary,  (a) is trivially satisfied, (b) follows from the fact that 
$$
H^0(C, T_{\ov S|C}\otimes\mcO_C(-C))=\ker\{H^0(C, N_{C/\ov S}\otimes T_C)\overset \alpha\longrightarrow H^1(C, T_C^2)\}
$$
(and $\alpha$ is injective), and (c) follows from the BNP property of $C$.

\end{proof}

\bibliographystyle{abbrv}

\bibliography{bibliografia}

\end{document}